\documentclass[a4,11pt]{amsart}

\usepackage{latexsym}
\usepackage{amsmath}
\usepackage{amssymb}
\usepackage{amsthm}
\usepackage[mathcal]{eucal}
\usepackage{amsmath, amssymb,enumerate}
\usepackage{mathrsfs}
\usepackage{enumerate}
\usepackage[colorlinks=true, pdfstartview=FitV, linkcolor=blue, citecolor=blue,
 urlcolor=blue]{hyperref}
\usepackage{color}
\allowdisplaybreaks

 \makeatletter \@addtoreset{equation}{section}

\makeatother \makeatletter

\newtheorem{theorem}{Theorem}
\newtheorem{proposition}{Proposition}
\newtheorem{lemma}{Lemma}
\newtheorem{remark}{Remark}

\newcommand{\R}{\mathbb{R}}

\newcommand{\C}{\mathbb{C}}

\newcommand{\N}{\mathbb{N}}

\allowdisplaybreaks
\begin{document}
\title[Elliptic operators with drift]{Elliptic operators with unbounded diffusion, drift and potential terms}
\author{S.E. Boutiah}
\address{Department of Mathematics, University Ferhat Abbas Setif-1, Setif 19000, Algeria.}
\email{sallah\_eddine.boutiah@yahoo.fr}
\author{F. Gregorio}
\address{FernUniversit\"at in Hagen
Fakult\"at Mathematik und Informatik
Lehrgebiet Analysis, 58084 Hagen, Germany}
\email{fgregorio@unisa.it}
\author{A. Rhandi}
\address{Dipartimento di Ingegneria dell'Informazione, Ingegneria Elettrica e Matematica Applicata, Universit\`a degli
Studi di Salerno, Via Giovanni Paolo II, 132, I 84084 FISCIANO (Sa), Italy.}
\email{arhandi@unisa.it}
\author{C. Tacelli}
\address{Dipartimento di Ingegneria dell'Informazione, Ingegneria Elettrica e Matematica Applicata, Universit\`a degli
Studi di Salerno, Via Giovanni Paolo II, 132, I 84084 FISCIANO (Sa), Italy.}
\email{ctacelli@unisa.it}

\maketitle

\begin{abstract}
We prove that the realization $A_p$ in $L^p(\R^N),\,1<p<\infty$, of the elliptic operator
$A=(1+|x|^{\alpha})\Delta+b|x|^{\alpha-1}\frac{x}{|x|}\cdot \nabla-c|x|^{\beta}$ with domain
$D(A_p) =\{ u \in W^{2,p}(\mathbb{R}^N)\, |\,  Au \in L^p(\mathbb{R}^N)\}$ generates a strongly continuous analytic semigroup $T(\cdot)$ provided that
$\alpha >2,\,\beta >\alpha -2$ and any constants $b\in \R$ and $c>0$. This generalizes the recent results in
\cite{AC-AR-CT} and in \cite{Me-Sp-Ta}. Moreover we show that $T(\cdot)$ is consistent, immediately compact and ultracontractive.
\end{abstract}


\section{Introduction}
Starting from the 1950's, the theory of linear second order elliptic operators with bounded coefficients has widely been studied.
In recent years there has been a surge of activity focused on the case of unbounded coefficients. Let us recall some recent results concerning elliptic operators having polynomial coefficients.

In (\cite{G-S}, \cite{G-S 2}) it is proved that, under suitable assumptions on $\alpha$, the operator $(s+ \vert x \vert^{\alpha})\Delta$, $s=0,1$, generates an analytic semigroup in $L^{p}(\mathbb{R}^{N})$. In \cite{Luca - Abde} (resp. \cite{AC-AR-CT}) the generation of an analytic semigroup of the $L^p$-realization of the Schr\"odinger-type operators $(1+ \vert x \vert^{\alpha})\Delta- \vert x \vert^{\beta}$ in $L^{p}(\mathbb{R}^{N})$ for $\alpha \in [0,2)$ and $\beta>2$ (resp. $\alpha>2$, $\beta>\alpha-2$) is obtained. In  \cite{Luca - Abde} some estimates for the associated heat kernel are provided.

 Under suitable growth assumptions on the functions $q$ and $F$, it is proved in \cite{F-L} that the operator $\mathcal{A}u = q \Delta u + F \cdot \nabla u$ admits realizations generating analytic semigroups in $L^{p}(\mathbb{R}^{N})$ for any $p \in [1,+\infty]$ and in $C_{b}(\mathbb{R}^{N})$.

 Concerning the operator $( 1+ \vert x \vert^{\alpha})\Delta + b\vert x \vert^{\alpha-2}x\cdot \nabla $, G. Metafune et al. in \cite{Me-Sp-Ta} showed the generation of an analytic semigroup in $L^{p}(\mathbb{R}^{N})$ in the case where
 $\alpha>2$, $\frac{N}{N-2+b}<p<\infty$, and the semigroup is contractive if and only if $p\ge \frac{N+\alpha-2}{N-2+b}$.  Domain characterization and spectral properties as well as kernel estimates have been also proved. More recently in \cite{GM-NO-MS-CP} the authors showed that the operator $L = \vert x \vert^{\alpha}\Delta + b\vert x \vert^{\alpha-2}x\cdot \nabla - c\vert x \vert^{\alpha-2}$ generates a strongly continuous semigroup in $L^p(\mathbb{R}^N)$ if and only if $s_1+\min\{0,2-\alpha\}<\frac{N}{p}<s_2+\max\{0,2-\alpha\}$, where $s_i$ are the roots of the equation $c+s(N-2+b-s)=0$. Moreover the domain of the generator is also characterized.
 \\
 At this point it is important to note that the techniques used in \cite{GM-NO-MS-CP} are completely different from our and lead to results which are not comparable with our case ($\beta >\alpha -2$).

 In this paper, we are  interested in studying quantitative and qualitative properties in $L^p(\R^N),\,1<p<\infty$, of the elliptic operator
\begin{equation}\label{eq:operator-A}
 Au(x)=q(x)\Delta u(x)+b|x|^{\alpha-1}\frac{x}{|x|}\cdot\nabla u-V(x)u(x), \quad x\in\mathbb{R}^N,
\end{equation}
where $\alpha>2,\,\beta>\alpha-2$, $q(x)=(1+|x|^{\alpha})$ and $V(x)=c|x|^\beta$, with $b\in \R$ and $c>0$.

Let $A_p$ be the realization of $A$ in $L^p(\mathbb{R}^N)$ endowed with the maximal domain
\begin{equation}
D_{p,max}(A)=\{ u\in L^p(\mathbb{R}^N) \cap W^{2,p}_{loc}(\mathbb{R}^N)\,:\;Au\in L^p(\mathbb{R}^N) \}.
\end{equation}
After proving a priori estimates, we deduce that the maximal domain $D_{p,max}(A)$ of the operator $A$  coincides with
\begin{eqnarray*}
D_p(A):=\{ u\in W^{2,p}(\mathbb{R}^N)\;:\; Vu,\ (1+|x|^{\alpha-1})\nabla u,\ (1+|x|^{\alpha}) D^2u\in L^p(\mathbb{R}^N)\}.
\end{eqnarray*}
So, we show in the main result of this paper that, for any $1<p<\infty$, the realization $A_p$ of $A$ in $L^p(\mathbb{R}^N)$, with
domain $D_p(A)$
generates a positive strongly continuous and analytic semigroup $(T_p(t))_{t\ge 0}$ for $p\in (1,\infty)$.
This semigroup is also consistent, irreducible, immediately compact and ultracontractive.

The paper is divided as follows. In section \ref{section2} we recall the solvability of the elliptic and parabolic problems in spaces of continuous functions. In Section \ref{section3} we introduce the definition of the reverse H\"older class and recall some results given in \cite{shen-1995} and in \cite{AC-AR-CT} to study the solvability of the elliptic problem in $L^p(\R^N)$. In section \ref{section4} we prove that the maximal domain of the operator $A$ coincides with the weighted Sobolev space $D_p(A)$, and we state and prove the main result of this paper.\\

{\large \bf Notation.}
In general we use standard notations for function spaces. We denote by $L^p(\mathbb{R}^N)$ and $W^{2,p}(\mathbb{R}^N)$ the standard $L^p$ and Sobolev spaces, respectively. Following the notation in \cite{AC-AR-CT}, for any $k\in\N\cup\{\infty\}$ we denote by $C^{k}_c(\mathbb{R}^N)$ the set of all
functions $f: \mathbb{R}^N \to\mathbb{R}$ that are continuously differentiable in $\mathbb{R}^N$ up to $k$-th order and have compact support (say ${\rm supp}(f)$). The space $C_b(\mathbb{R}^N)$ is the set of all bounded and continuous functions $f:\mathbb{R}^N\to
\mathbb{R}$, and we denote by $\|f\|_{\infty}$ its sup-norm, i.e., $\|f\|_{\infty}=\sup_{x\in\mathbb{R}^N}|f(x)|$. We use also the space
$C_0(\mathbb{R}^N):=\{f\in C_b(\mathbb{R}^N): \lim_{|x|\to \infty}f(x)=0\}.$
If $f$ is smooth enough we set
\begin{eqnarray*}
|\nabla f(x)|^2=\sum_{i=1}^N|D_if(x)|^2,\qquad
|D^2f(x)|^2=\sum_{i,j=1}^N|D_{ij}f(x)|^2.
\end{eqnarray*}
For any $x_0\in\mathbb{R}^N$ and any $r>0$ we denote by $B(x_0,r)\subset\mathbb{R}^N$ the open ball, centered at $x_0$ with radius $r$.
We simply write $B(r)$ when $x_0=0$.
The function $\chi_E$ denotes the characteristic function of the set $E$,
i.e., $\chi_E(x)=1$ if $x\in E$, $\chi_E(x)=0$ otherwise.
Finally, by $x\cdot y$ we denote the Euclidean scalar product of the vectors $x,y\in\mathbb{R}^N$.
\section{Solvability in $C_{0}(\mathbb{R}^{N})$}\label{section2}
In this short section we briefly recall some
properties of the elliptic and parabolic problems associated with $A$ in spaces of continuous functions.

Let us first consider the operator $A$ on $C_b(\R^N)$ with its maximal domain
$$ D_{max}(A)= \{ u \in C_{b}(\mathbb{R}^N)\cap W_{loc}^{2,p}(\mathbb{R}^N) \quad\text{for all}\quad 1\le p<\infty : Au \in C_{b}(\mathbb{R}^N) \}.$$

It is known, cf. \cite[Chapter 2, Section 2]{Lo-Be}, that to the associated parabolic problem
\begin{equation}     \label{problem}
\left\{
\begin{array}{ll}
u_t(t,x)=Au(t,x)& x\in\R^N,\ t>0, \\
u(0,x)=f(x)   &x\in\R^N\;, \end{array}\right.
\end{equation}
where $f\in C_b(\R^N)$, one can associate a semigroup $(T(t))_{t\ge 0}$ of bounded operator
in $C_b(\R^N)$ such that $u(t,x)=T(t)f(x)$ is a solution of \eqref{problem} in the following sense:
$$u\in C([0,+\infty)\times \R^N)\cap C_{loc}^{1+\frac{\sigma}{2},2+\sigma}((0,+\infty)\times \R^N)$$
and $u$ solves \eqref{problem} for any $f\in C_b(\R^N)$ and some $\sigma \in (0,1)$.
Moreover, in our case the solution is unique. This can be seen by proving the existence of a Lypunov function for $A$, i.e.,
 a positive function  $\varphi(x)\in C^2(\R^N)$ such that $\lim_{|x|\to \infty}\varphi(x)=+\infty$
and $A\varphi-\lambda \varphi\leq 0$ for some $\lambda>0$.
\begin{proposition}
Assume that $\alpha>2$ and $\beta>\alpha-2$. Let $\psi = 1 + \vert x\vert^{\gamma}$ where $\gamma>2$ then there exists a constant $C>0$ such that
$$ A\psi \le C\psi .$$
\end{proposition}
\begin{proof}
An easy computation gives
\begin{align*}
 A\psi & = \gamma(N+\gamma-2)(1+\vert x\vert^{\alpha})\vert x \vert^{\gamma-2} + b\gamma\vert x\vert^{\alpha-1}\vert x\vert^{\gamma-2} - c(1+\vert x\vert^{\gamma})\vert x\vert^{\beta} \\
&  \le \{ \gamma(N+\gamma-2) +  \vert b \vert \gamma \}(1+\vert x\vert^{\alpha})\vert x \vert^{\gamma-2}  - c(1+\vert x\vert^{\gamma})\vert x\vert^{\beta}.
\end{align*}
Since $\beta>\alpha-2$, it follows that there exists $C>0$ such that
$$ \{ \gamma(N+\gamma-2) +  \vert b \vert \gamma \}(1+\vert x\vert^{\alpha})\vert x \vert^{\gamma-2} \le c(1 + \vert x\vert^{\gamma}) \vert x\vert^{\beta} + C(1 + \vert x\vert^{\gamma}).$$
Thus, $\psi$ is a Lyapunov function for $A$.
\end{proof}

As in \cite{AC-AR-CT} one can prove the following result.
\begin{proposition}
Assume that $N>2,\,\alpha>2$ and $\beta >\alpha -2$. Then the semigroup $(T(t))$ is generated by $ (A, D_{max}(A))\cap C_{0}(\mathbb{R}^{N})$ and maps $C_{0}(\mathbb{R}^{N})$ into $C_{0}(\mathbb{R}^{N})$.
\end{proposition}

\begin{proof}
We split the proof into two steps.

{\em Step 1.} Here we assume that $b> 2-N$.

Let $L_{0}$ be the operator defined by $ L_{0} := q(x)\Delta +b \vert x \vert^{\alpha - 1}\frac{x}{\vert x \vert} \cdot \nabla$, and as a result of \cite[Proposition 2.2 (ii)]{Me-Sp-Ta}, we have that the minimal semigroup $(S(t))$ is generated by $(L_{0},D_{max}(L_{0}))\cap C_{0}(\mathbb{R}^{N})$. Moreover the resolvent and the semigroup map $C_{b}(\mathbb{R}^{N})$ into $C_{0}(\mathbb{R}^{N})$ and are compact.\\
Set $ v(t,x) = S(t)f(x)$ and $ u(t,x) = T(t)f(x)$ for $ t > 0, x \in \mathbb{R}^{N}$ and $  0\le f \in C_{b}(\mathbb{R}^{N})$.
Then the function $ w(t,x) = v(t,x) - u(t,x)$ solves
$$
\begin{cases}
&w _{t}(x,t) = L_{0}w(t,x) + V(x)u(t,x),\ \  t>0,\\
&w(0,x) = 0 \qquad x \in \mathbb{R}^{N}.
\end{cases}
$$
So, applying \cite[Theorem 4.1.3]{Lo-Be}, we have $w\ge 0$ and hence $ T(t)\le S(t)$. Thus, $T(t)\mathbf{1} \in C_{0}(\mathbb{R^{N}})$, for any $t>0$ (see \cite[Proposition 2.2 (iii)]{G-S}). Therefore $T(t)$ is compact and $C_{0}(\mathbb{R}^{N})$ is invariant for $T(t)$ (cf. \cite[Theorem 5.1.11]{Lo-Be}).

{\em Step 2.} Assume now that $b\le  2-N$.

Let $f \in C_0(\mathbb{R}^N)$. Since $C_c^\infty(\R^N)$ is dense in $C_0(\R^N)$, there is a sequence $(f_n)\subset C_c^\infty(\R^N)$ such that $\lim_{n\to \infty}\|f_n-f\|_\infty=0$.

On the other hand, it follows from Theorem \ref{generation} that the operator $A_p$ with domain $D_{p,max}(A)$ generates an analytic semigroup $T_p(t)$ in $L^p(\mathbb{R}^N)$, and $D_p(A)$ is continuously embedded into $W^{2,p}(\mathbb{R}^N)$. Hence, by Theorem \ref{risolvente} and Sobolev's embedding theorem, $T(t)f_n=T_p(t)f_n \in D_p(A)\subset W^{2,p}(\mathbb{R}^N)\hookrightarrow C_0(\mathbb{R}^N)$ for $p> \frac{N}{2}$. Since $f_n \rightarrow f$ uniformly, it follows that $T(t)f_n \rightarrow T(t)f$ uniformly. Hence $T(t)f \in C_0(\mathbb{R}^N)$.
\end{proof}
\begin{remark}
As one sees from the proof of the above proposition, in the case where $b>2-N$ the semigroup $(T(t))$ generated by $ (A, D_{max}(A))\cap C_{0}(\mathbb{R}^{N})$ is compact.
\end{remark}

\section{Solvability of $\lambda u-Au=f$ in $L^p(\mathbb{R}^N)$}\label{section3}
In the previous section we have proved the existence and uniqueness of the elliptic and parabolic problems in $C_{0}(\mathbb{R}^{N})$. In this section we study the solvability of the equation $\lambda u-A_pu=f$ for $\lambda>\lambda_0$, where $\lambda_0$ is a suitable positive constant.

Let $f\in L^p(\mathbb{R}^N)$ and consider  the equation
\begin{equation}\label{eq:risolv}
\lambda u-Au=f.
\end{equation}
Let $\phi=(1+|x|^\alpha)^{b/\alpha}$ and set  $u=\frac{v}{\sqrt \phi}$. A simple computation  gives
\begin{align}\label{eq:lambdau-Au}
& \lambda u -Au=\frac{(1+|x|^\alpha)}{\sqrt \phi}\left[-\Delta v+Uv+ \frac{V+\lambda}{1+|x|^\alpha}v\right],
\end{align}
where
\begin{align*}
U=-\frac{1}{4}\left|\frac{\nabla \phi}{\phi}\right|^2+\frac{1}{2}\frac{\Delta \phi}{\phi}.
\end{align*}
Then solving \eqref{eq:risolv}
is equivalent to solve
\begin{equation}\label{eq:risolv-equiv}
-Hv=\frac{\sqrt \phi}{1+|x|^\alpha}f,
\end{equation}
where $H$ is the Schr\"odinger operator defined by
\[
 H=\Delta -U-\frac{V+\lambda}{1+|x|^\alpha}.
\]
If we denote by $G(x,y)$ the Green function of $H$, a solution of \eqref{eq:risolv-equiv} is given by
\[
 v(x)=\int_{\mathbb{R}^N} G(x,y) \frac{\sqrt {\phi(y)}}{1+|y|^\alpha}f(y)dy,
\]
and hence a solution of \eqref{eq:risolv} should be
\begin{equation}\label{eq:def-L}
u(x)=Lf(x):=\frac{1}{\sqrt {\phi (x)}}\int_{\mathbb{R}^N} G(x,y) \frac{\sqrt {\phi(y)}}{1+|y|^\alpha}f(y)dy.
\end{equation}

First, we have to show that $L$ is a bounded operator in $L^p(\mathbb{R}^N)$.
For this purpose we need to estimate $G.$

We focus our attention to the operator $H$. Evaluating the potential $\mathcal {V}=U+\frac{V+\lambda}{1+|x|^\alpha}$, it follows that
\begin{align*}
\mathcal {V} &=\frac{|x|^{2\alpha-2}}{(1+|x|^\alpha)^2}\left(\frac{b^2}{4}-\frac{b\alpha}{2} \right)
  +\frac{|x|^{\alpha-2}}{1+|x|^\alpha}\frac{b}{2}(N+\alpha-2)
  +\frac{c|x|^\beta +\lambda}{1+|x|^{\alpha}}
\\
&=
\left(\frac{1}{1+|x|^\alpha}-\frac{1}{(1+|x|^\alpha)^2}\right)|x|^{\alpha -2}\left(\frac{b^2}{4}-\frac{b\alpha}{2} \right)+\frac{|x|^{\alpha-2}}{1+|x|^\alpha}\frac{b}{2}(N+\alpha-2)
  +\frac{c|x|^\beta +\lambda}{1+|x|^{\alpha}}\\
&=\frac{|x|^{\alpha-2}}{1+|x|^\alpha}
  \left(\frac{b^2}{4}+b\left(\frac{N-2}{2}\right)\right)
  +\frac{|x|^{\alpha-2}}{(1+|x|^\alpha)^2}
  \left(
      -\frac{1}{4}b^2+\frac{1}{2}b\alpha
      \right)
  +\frac{c|x|^\beta}{1+|x|^{\alpha}}
	+\frac{\lambda}{1+|x|^\alpha}.
\end{align*}
We can choose $\lambda_0>0$ such that for every $\lambda\geq \lambda_0$ the potential $\mathcal V$ is positive.
Indeed, since $\beta>\alpha-2$ the function
\[
\frac{|x|^{2\alpha-2}}{(1+|x|^\alpha)}\left(\frac{b^2}{4}-\frac{b\alpha}{2} \right)
  +{|x|^{\alpha-2}}\frac{b}{2}(N+\alpha-2)
  +{c|x|^\beta}
\]
has a nonpositive minimum $\mu$ in $\mathbb{R}^N$. So, one takes $\lambda_0>-\mu.$\\
On the other hand, since $\mathcal V(0)=\lambda>0$ and $\mathcal V$ behaves like  $|x|^{\beta-\alpha}$ as $|x|\to \infty$  we have the following estimates
\begin{align}\label{eq:stime-v}
&C_1(1+|x|^{\beta-\alpha})\leq \mathcal V \leq C_2 ({1+|x|^{\beta-\alpha}})\quad\text{ if }\beta\geq \alpha ,\\
&C_3\frac{1}{1+|x|^{\alpha-\beta}}\leq \mathcal V \leq C_4 \frac{1}{1+|x|^{\alpha-\beta}}\quad\text{ if }\alpha-2<\beta<\alpha \nonumber
\end{align}
for some positive constants $C_1,C_2,C_3,C_4$.

At this point we can use bounds of $G$ obtained by \cite{shen-1995} in the case of positive potentials belonging to the reverse H\"older class $B_q$ for some $q\geq N/2$.

We recall that a nonnegative locally $L^q$-integrable function $V$ on $\mathbb{R}^N$ is said to be in $B_q,\,1 < q < \infty$, if there exists $C > 0$ such that the reverse
H\"older inequality
\[
\left( \frac{1}{|B|}\int_BV^q(x) dx \right) ^{1/q}\leq C\left( \frac{1}{|B|}\int_BV(x) dx \right)
\]
holds for every ball $B$ in $\mathbb{R}^N$. A nonnegative function $V\in L_{\rm loc}^\infty(\mathbb{R}^N)$ is in
$B_\infty $ if
\[
\|V\|_{L^\infty(B)}\leq C\left( \frac{1}{|B|}\int_BV(x) dx \right)
\]
for every ball $B$ in $\mathbb{R}^N$.

As regards the potential $\mathcal V$,
by \eqref{eq:stime-v}  and since $\beta-\alpha>-2$, we have that $\mathcal V\in B_{N/2}$. So, it follows from \cite[Theorem 2.7]{shen-1995} that for any $k\in \N$ there is a constant $C_k>0$ such that
\begin{equation}\label{eq:stima-G-shen}
|G(x,y)|\leq\frac{C_k}{\left( 1+m(x)|x-y| \right)^k}\cdot \frac{1}{|x-y|^{N-2}},\quad x,\,y\in \mathbb{R}^N,
\end{equation}
where the auxiliary function $m$ is defined by
\begin{equation}\label{eq:def-m}
\frac{1}{m(x)}:=\sup_{r>0}\left\{r\,:\,\frac{1}{r^{N-2}} \int_{B(x,r)}\mathcal V(y)dy\leq 1 \right\},\quad x\in \mathbb{R}^N.
\end{equation}

In \cite{AC-AR-CT}  a lower bound for the auxiliary function
associated to the potential $\tilde V=\frac{|x|^\beta}{1+|x|^{\alpha}}$ was obtained.
Since  $\mathcal V\geq C_1\tilde V$ for some positive constant $C_1$, we have $m(x)\geq \widetilde{m}(x)$,
where $\frac{1}{\widetilde{m}(x)}:=\sup_{r>0}\left\{r\,:\,\frac{1}{r^{N-2}} \int_{B(x,r)} C_1\tilde V(y)dy\leq 1 \right\}$.
Replacing $\tilde V$ with $C_1\tilde V$ in \cite[Lemma 3.1, Lemma 3.2]{AC-AR-CT} we obtain
$\widetilde{m}(x)\geq  C_2 \left( 1+|x|\right)^{\frac{\beta-\alpha}{2}}$. So we have
\begin{lemma}\label{pr:mx-beta-le-alpha}
Let $\alpha-2<\beta$. There exists $C=C(\alpha,\beta,N)$ such that
\begin{equation}\label{eq:stima-m-beta-le-alpha}
m(x)\geq  C \left( 1+|x|\right)^{\frac{\beta-\alpha}{2}}.
\end{equation}
\end{lemma}
Finally by \eqref{eq:stima-G-shen} and the previous lemma we can estimate the Green function $G$
\begin{lemma}
Let $G(x,y)$ denotes the Green function of the Schr\"odinger operator $H$ and assume
that
$\beta>\alpha-2$. Then
\begin{equation}\label{eq:stima-G}
G(x,y)\leq C_k\, \frac{1}{ 1+|x-y|^k\,\left( 1+|x| \right)^{\frac{\beta-\alpha}{2} k }}\;\frac{1}{|x-y|^{N-2}}, \quad
x,y\in\mathbb{R}^N
\end{equation}
for any $k>0$ and some constant $C_k>0$ depending on $k$.
\end{lemma}

We can prove now the boundedness in $L^p(\mathbb{R}^N)$ of the operator $L$ given by \eqref{eq:def-L}

\begin{lemma} \label{chius2} Assume that $\alpha > 2$, $N>2$  and $\beta>\alpha-2$. Then
there exists a positive constant $C=C(\lambda)$ such that for every $0\leq \gamma\leq \beta$ and
$f\in L^p(\mathbb{R}^N)$
\begin{equation}\label{eq:stima-L}
\||x|^\gamma Lf\|_{p}\leq  C\|f\|_{p}.
\end{equation}
\end{lemma}
\begin{proof}
Recall the function $\phi(x)=(1+|x|^\alpha)^{b/\alpha}$. Let $\Gamma(x,y)=\sqrt {\frac{\phi (y)}{\phi(x)}}\frac{G(x,y)}{1+|y|^\alpha}$, $f\in L^p(\mathbb{R}^N)$ and
\[
u(x)=\int_{\mathbb{R}^N}\Gamma (x,y)f(y)dy,\quad x\in \R^N .
\]
We have to show that
\[
\||x|^\gamma u\|_p\leq C\|f\|_p.
\]
By setting $\Gamma_0=\frac{G(x,y)}{1+|y|^\alpha}$,
we have $\Gamma(x,y)=\left( \frac{1+|y|^\alpha}{1+|x|^\alpha} \right)^{b/(2\alpha)}\Gamma_0(x,y)$. Moreover
if  we set $L_0f(x):=\int_{\mathbb{R}^N}\Gamma_0 (x,y)f(y)dy,\,x\in \R^N$, then
\cite[Lemma 3.4]{AC-AR-CT} gives
\begin{equation}\label{eq:stima-u_0}
\||x|^\gamma L_0f\|_p\leq C\|f\|_p.
\end{equation}

For $x\in \R^N$ let us consider the regions $E_1:=\{|x-y|\leq \frac{1}{2}(1+|y|) \}$
and $E_2:=\{|x-y|> \frac{1}{2}(1+|y|) \}$ and write
\[
u(x)=\int_{E_1} \Gamma (x,y)f(y)dy+\int_{E_2} \Gamma (x,y)f(y)dy=:u_1(x)+u_2(x)\;.
\]
In $E_1$ we have $1+|y|\leq 1+|x|+|x-y|\leq 1+|x|+\frac{1}{2}(1+|y|)$ and hence $\frac{1}{2}(1+|y|)\leq 1+|x|$. Thus,
\[
 \frac{1+|x|}{1+|y|}\leq \frac{1+|x-y|+|y|}{1+|y|}\leq \frac{3}{2}\text{ and }\frac{1+|y|}{1+|x|}\leq 2\,.
\]
Therefore there are constants $C,\,\widetilde{C}>0$ such that $\left(\frac{1+|y|^\alpha}{1+|x|^\alpha}\right)^{b/(2\alpha)}\le \widetilde{C}\left(\frac{1+|y|}{1+|x|}\right)^{b/2}\leq C2^{|b|/2}$ and
$\Gamma(x,y)\leq C \Gamma_0(x,y)$ in $E_1$.
So, we have $$|u_1(x)|\leq C\int_{\mathbb{R}^N}\Gamma_0(x,y)|f(y)|dy=CL_0(|f|)(x).$$ By \eqref{eq:stima-u_0} it follows that
$\| |x|^\gamma u_1\|_p\leq C \|f\|_p$.

As regards  the region $E_2$, we have, by H\"older's inequality,
\begin{align}\label{eq:u-2-holder}
& \left| |x|^\gamma u_2(x) \right|
  \leq |x|^\gamma \int_{E_2}\Gamma(x,y)|f(y)|dy=\int_{E_2}  \left( |x|^\gamma \Gamma(x,y)\right) ^{\frac{1}{p'}}
    \left( |x|^\gamma \Gamma(x,y)\right)^{\frac{1}{p}}|f(y)|dy\nonumber \\
&\quad\leq
  \left(\int_{E_2}|x|^\gamma \Gamma(x,y)dy \right)^{\frac{1}{p'}}
  \left(\int_{E_2}|x|^\gamma \Gamma(x,y)|f(y)|^p dy\right)^{\frac{1}{p}}\;.
\end{align}
We propose to estimate first $\int_{E_2}|x|^\gamma\Gamma(x,y)dy$.
In $E_2$ we have
$1+|y|\leq 2|x-y|$ and
$1+|x|\leq  1+|y|+|x-y|\leq 3|x-y|$,
then
\[
\left(\frac{1+|y|^\alpha}{1+|x|^\alpha}\right)^{b/(2\alpha)}\le \widetilde{C}\left(  \frac{1+|y|}{1+|x|}\right)^{b/2}\leq C|x-y|^{|b|/2}.
\]
From \eqref{eq:stima-G} and by the symmetry of $G$ it follows that
\begin{eqnarray*}
|x|^\gamma \Gamma (x,y) &=& |x|^\gamma \left(\frac{1+|y|^\alpha}{1+|x|^\alpha}\right)^{b/(2\alpha)} \frac{G(x,y)}{1+|y|^\alpha}\\
&\leq & C |x|^\gamma G(x,y)|x-y|^{|b|/2}\\
&\leq & C \frac{1+|x|^\beta}{|x-y|^k\left( 1+|y| \right)^{k\frac{\beta-\alpha}{2}}}\frac{1}{|x-y|^{N-2-|b|/2}}\\
&\leq & C \frac{1}{|x-y|^{k-\beta+N-2-|b|/2}}
    \frac{1}{\left( 1+|y| \right)^{k\frac{\beta-\alpha}{2}}},\quad y\in E_2.
\end{eqnarray*}
For every $k>\beta-N+2+|b|/2$, taking into account that $\frac{1}{|x-y|}\leq 2 \frac{1}{1+|y|}$, we get
\[
|x|^\gamma \Gamma (x,y) \leq
    C\frac{1}{\left( 1+|y| \right)^{k\frac{\beta-\alpha+2}{2} +N-2-\beta-|b|/2}}\;.
\]
Since $\beta-\alpha+2>0$ we can choose $k$ such that $\frac{k}{2}(\beta-\alpha +2)+N-2-\beta-|b|/2>N$,
then there is a constant $C_1>0$ such that
\begin{align*}
\int_{E_2}|x|^\gamma \Gamma(x,y)dy\leq
C\int_{\mathbb{R}^N} \frac{1}{(1+|y|)^{\frac{k}{2}(2+\beta-\alpha)+N-2-\beta-|b|/2}}dy\le C_1.
\end{align*}
Moreover by
\eqref{eq:stima-G}
as above we have
\begin{eqnarray*}
|x|^\gamma \Gamma (x,y) &\leq & C|x|^\gamma G(x,y)|x-y|^{|b|/2}\\
  &\leq & C\frac{1+|x|^\beta}{|x-y|^k\left( 1+|x| \right)^{k\frac{\beta-\alpha}{2}}}\frac{1}{|x-y|^{N-2-|b|/2}}\\
  &\leq & C \frac{1}{|x-y|^{k-\beta +N-2-|b|/2}}
    \frac{1}{\left( 1+|x| \right)^{k\frac{\beta-\alpha}{2}}}\;.
\end{eqnarray*}
Taking into account that $\frac{1}{|x-y|}\leq 3\frac{1}{1+|x|}$, arguing as above we obtain
\begin{align}\label{eq:vudx}
\int_{E_2}|x|^\gamma \Gamma(x,y)dx
\leq C_2
\end{align}
for some constant $C_2>0$.
Hence \eqref{eq:u-2-holder} implies
\begin{equation}\label{eq:V-2-u-2}
\left| |x|^\gamma u_2(x) \right|^p \leq C_1^{p-1}\int_{E_2}|x|^\gamma \Gamma(x,y)|f(y)|^pdy.
\end{equation}
Thus, by \eqref{eq:V-2-u-2} and \eqref{eq:vudx}, we have
\begin{align*}
&\||x|^\gamma u_2\|^p_p\leq C_1^{p-1} \int _{\mathbb{R}^N} \int _{\mathbb{R}^N}|x|^\gamma\Gamma(x,y)\chi_{\{|x-y|>\frac{1}{2}(1+|y|)\}}(x,y)|f(y)|^p dydx\\
&\quad = C_1^{p-1}\int _{\mathbb{R}^N} |f(y)|^p \left(\int _{E_2}|x|^\gamma \Gamma(x,y)dx\right)dy\leq C_1^{p-1}C_2\|f\|_p^p\;.
\end{align*}
\end{proof}

Here and in Section 4 we will need the following covering result, see \cite[Proposition 6.1]{cup-for}.

\begin{proposition} \label{pr:covering}
Given a covering $\mathcal{F}= \{B(x,\rho(x))\}_{x\in \mathbb{R}^N}$ of $\R^N$, where $\rho:\mathbb{R}^N \to \mathbb{R}_+$ is a Lipschitz continuous function with Lipschitz constant $k<1/2$, there exists a countable subcovering $\{B(x_n,\rho(x_n))\}_{n\in\N}$ of $\mathbb{R}^N$ and $\zeta=\zeta (N,k)\in \N$ such that at most $\zeta$ among the
double balls
$\{B(x_n,2\rho(x_n))\}_{n\in \N}$ overlap.
\end{proposition}

We propose now to characterize the domain $D_{p,max}(A)$.
\begin{proposition}\label{max-reg}
Assume that $N>2$, $\alpha>2$ and $\beta>\alpha-2$.
For $1<p<\infty$ the following holds
$$D_{p,max}(A)=\{u\in W^{2,p}(\R^N): Au\in L^p(\R^N)\}.$$
\end{proposition}
\begin{proof}
It suffices to prove that $D_{p,max}(A)\subset \{u\in W^{2,p}(\R^N): Au\in L^p(\R^N)\}$. Let $u\in D_{p,max}(A)$. Then $f:=Au\in L^p(\R^N)$. This implies that
$$\widetilde{A}u:=\Delta u+b\frac{|x|^{\alpha -2}}{1+|x|^\alpha}x\cdot \nabla u-\frac{c|x|^\beta}{1+|x|^\alpha}u=\frac{f}{1+|x|^\alpha}\in L^p(\R^N).$$
If $\beta \le \alpha$ then the potential $\widetilde{V}(x):=\frac{c|x|^\beta}{1+|x|^\alpha}$ is bounded and by standard regularity results for uniformly elliptic operators with bounded coefficients we deduce that $u\in W^{2,p}(\R^N)$.

Let us assume now that $\beta >\alpha$. Then $\widetilde{V}\in B_q$ for all $q\in (1,\infty)$. So, by \cite[Theorem 1.1 and Corollary 1.3]{Au-Ben},
we have that $D_{p,max}(\Delta -\widetilde{V})=W^{2,p}(\R^N)\cap D_{p,max}(\widetilde{V})$ and the following estimate holds
\begin{equation}\label{au-ben}
\|\widetilde{V}f\|_p+ \|\Delta f\|_p\leq C\|\Delta f-\widetilde{V}f\|_p
\end{equation}
for all $f\in D_{p,max}(\Delta -\widetilde{V})$ with a constant $C$ independent of $f$.

Fix now $x_0\in \R^N$ and $R\ge 1$. We propose to prove the following interior estimate
\begin{equation}\label{eq:interior-estimate1}
 \|\Delta u\|_{L^p(B(x_0,\frac{R}{2}))}\leq C\left(\|\widetilde{A} u\|_{L^p(B(x_0,R))}+\|u\|_{L^p(B(x_0,R))}\right)
\end{equation}
with a constant $C$ independent of $u$ and $R$. To this purpose take $\sigma \in (0,1)$ and set $\sigma ':=\frac{\sigma +1}{2}$. Consider a cutoff function
$\vartheta\in C_c^{\infty }(\mathbb{R}^N)$ such that $0\leq \vartheta \leq 1$,
$\vartheta(x)=1$ for $x\in B(x_0,\sigma R)$, $\vartheta(x)=0$ for $x\in B^c(x_0,\sigma'R)$,
$\|\nabla \vartheta \|_{\infty}\leq \frac{C}{R(1-\sigma')}$ and $\|\Delta \vartheta\|_{\infty}\leq \frac{C}{R^2(1-\sigma')^2}$ with a constant $C$ independent of $R$.
\\
In order to simplify the notation we write $\|\cdot\|_{p,r}$ instead of $\|\cdot \|_{L^p\left( B(x_0,r) \right)}$. The function $v=u\vartheta$ belongs to $D_{p,max}(\Delta-\tilde V)$ and so by \eqref{au-ben} we have
\begin{eqnarray*}
& & \|\Delta u\|_{p,\sigma R}\leq \|\Delta v\|_{p,\sigma'R}\leq C \|\Delta v-\widetilde{V}v\|_{p,\sigma'R}\\
&\leq C & \left( \|\Delta v+F\cdot \nabla v-\widetilde{V}v\|_{p,\sigma'R} +\|F\cdot \nabla v\|_{p,\sigma'R}\right)\\
&\leq C & \left(
	    \| \widetilde{A}u \|_{p,\sigma'R}+2\|\nabla \vartheta\|_{\infty}\|\nabla u\|_{p,\sigma 'R}+\|\Delta\vartheta\|_{\infty}\|u\|_{p,\sigma' R}
	    +\|F\|_{\infty}\|\nabla \vartheta\|_{\infty}\|u\|_{p,\sigma'R}
	    \right.\\
& &\left. \,\,+\|F\|_{\infty}\|\nabla u\|_{p,\sigma' R}+\|F\|_{\infty}\|\nabla \vartheta \|_{\infty}\|u\|_{p,\sigma' R} \right)\\
&\leq C & \left(  \| \widetilde{A}u \|_{p,\sigma'R}+(\|\nabla \vartheta\|_{\infty}+\|F\|_\infty)\|\nabla u\|_{p,\sigma 'R}
	  +(\|\Delta\vartheta\|_{\infty}+\|\nabla \vartheta \|_\infty)\|u\|_{p,\sigma' R}\right)\\
&\leq C & \left(  \| \widetilde{A} u \|_{p,\sigma'R}+\frac{1}{R(1-\sigma')}\|\nabla u\|_{p,\sigma 'R}
	  +\frac{1}{R^2(1-\sigma')^2}\|u\|_{p,\sigma' R}\right),
\end{eqnarray*}
where $F(x):=b\frac{|x|^{\alpha -2}}{1+|x|^\alpha}x$ and $C$ a positive constant independent of $u$ and $R$, which may change from line to line. Multiplying the above estimate by $R^2(1-\sigma ')^2$ and taking into account that $1-\sigma=2(1-\sigma')$ we obtain
\[
R^2(1-\sigma)^2\|\Delta u\|_{p,\sigma R} \leq C\left( R^2\|\widetilde{A} u\|_{p,R}+R(1-\sigma')\|\nabla u\|_{p,\sigma'R}+\|u\|_{p,R} \right).
\]
So,
\begin{equation}\label{eq:sup-eq}
 \sup_{\sigma\in (0,1)}\left\{ R^2(1-\sigma)^2\|\Delta u \|_{p,\sigma R} \right\}\leq
 C\left( \sup_{\sigma\in (0,1)}\left\{R(1-\sigma)\|\nabla u\|_{p,\sigma R}\right\}+R^2\|\widetilde{A} u\|_{p,R}+\|u\|_{p,R} \right).
\end{equation}
Thus, by \cite[Theorem 7.28]{DG-NT}, for every $\gamma>0$ there exists $\sigma_\gamma\in (0,1)$ such that
\begin{eqnarray*}
 \sup_{\sigma\in (0,1)}\left\{R(1-\sigma)\|\nabla u\|_{p,\sigma R}\right\} &\leq & R(1-\sigma_\gamma)\|\nabla u\|_{p,\sigma_\gamma R}+\gamma \\
 &\le & \varepsilon R^2(1-\sigma_\gamma)^2\|\Delta u\|_{p,\sigma_\gamma R}+\frac{C}{\varepsilon}\|u\|_{p, R}+\gamma \\
 &\le & \varepsilon \sup_{\sigma\in (0,1)}\left\{R^2(1-\sigma)^2\|\Delta u\|_{p,\sigma R}\right\}+\frac{C}{\varepsilon}\|u\|_{p,R}+\gamma.
\end{eqnarray*}
Letting $\gamma \to 0$ we deduce that
\begin{equation}\label{eq:intepol-gamma}
 \sup_{\sigma\in (0,1)}\left\{R(1-\sigma)\|\nabla u\|_{p,\sigma R}\right\}
 \leq \varepsilon \sup_{\sigma\in (0,1)}\left\{R^2(1-\sigma)^2\|\Delta u\|_{p,\sigma R}\right\}+\frac{C}{\varepsilon}\|u\|_{p,R}.
\end{equation}
Putting \eqref{eq:intepol-gamma} into \eqref{eq:sup-eq} with a suitable choice of $\varepsilon$ we obtain
\[
 \sup_{\sigma\in (0,1)}\left\{ R^2(1-\sigma)^2\|\Delta u \|_{p,\sigma R} \right\}\leq
 C\left( R^2\|\widetilde{A} u\|_{p,R}+\|u\|_{p,R} \right).
\]
Hence \eqref{eq:interior-estimate1} follows since $(1-\frac{1}{2})^2R^2\|\Delta u \|_{p, \frac{R}{2}}\leq \sup_{\sigma\in (0,1)}\left\{ R^2(1-\sigma)^2\|\Delta u \|_{p,\sigma R} \right\}$.

To prove that $u\in W^{2,p}(\R^N)$ we consider a covering $\{B(x_n,R/2): n\in \N \}$ of $\R^N$ such that at most $\zeta$ among the doubled balls
$\{B(x_n,R): n\in \N \}$ overlap for some $\zeta(N)\in \N$, by Proposition \ref{pr:covering}. Applying \eqref{eq:interior-estimate1} with the ball $B(x_n,R/2)$ we obtain
\begin{eqnarray*}
\|\Delta u\|_p &\le & \sum_{n\in \N}\|\Delta u\|_{L^p(B(x_n,R/2))}\\
&\le & C \sum_{n\in \N}\left(\|\widetilde{A}u\|_{L^p(B(x_n,R))}+\|u\|_{L^p(B(x_n,R))}\right)\\
&\le & C\zeta \left(\|\widetilde{A}u\|_p+\|u\|_p\right).
\end{eqnarray*}
This ends the proof.
\end{proof}

We show now the invertibility of $\lambda-A_p$ in $D_{p,max}(A)$ for all $\lambda \ge \lambda_0$, where $\lambda_0>0$ is such that $\mathcal{V}\ge 0$ for all $\lambda \ge \lambda_0$.
\begin{theorem}
\label{invertibile}
Assume that $N>2,\,\alpha > 2$ and $\beta>\alpha-2$. Then $[\lambda_0,\infty)\subset \rho(A_p)$ and $(\lambda -A_p)^{-1}=L$ for all $\lambda \ge \lambda_0$.
Moreover there exists $C=C(\lambda)>0$ such that, for every
$0\leq \gamma\leq \beta $ and $\lambda\geq \lambda_0$, the following holds
\begin{equation}\label{eq:stima-apriori}
\| |\cdot |^{\gamma }u\|_{p}\leq  C\|\lambda u-A_pu\|_{p},\quad \forall u\in D_{p,max}(A)\;.
\end{equation}
\end{theorem}
\begin{proof}
 First we prove the injectivity of
$\lambda- A_p$ for $\lambda\geq \lambda_0$. Let $u\in D_{p,max}(A)$ such that $\lambda u-A_pu=0$. We have to distinguish two cases. The first one is when $b\le 0$. In this case,
by \eqref{eq:lambdau-Au} we have
$Hv=\Delta v-\mathcal V v=0$ with $v=u\sqrt{\phi}\in D_{p,max}(H)=W^{2,p}(\R^N)\cap D_{p,max}\left( \mathcal {V} \right)$,
(see \cite{okazawa} or \cite{Au-Ben}).
Then multiplying $Hv $ by $v|v|^{p-2}$ and integrating by part (see \cite{Me-Sp}) over $\mathbb{R}^N$, we have
\begin{eqnarray*}
0&=& \int_{\mathbb{R}^N} v|v|^{p-2} \Delta v \,dx-\int_{\mathbb{R}^N}\mathcal V |v|^pdx \\
&=& -(p-1)\int_{\mathbb{R}^N}|v|^{p-2}|\nabla v|^2dx-\int_{\mathbb{R}^N}\mathcal V |v|^pdx.
\end{eqnarray*}
Then we have $v\equiv 0$ and hence $u\equiv 0$.

The second case is when $b>0$. For this we multiply $\frac{1}{1+|x|^\alpha}(\lambda u-A_pu)$ by $u|u|^{p-2}$ and using the fact that $u\in W^{2,p}$, by Proposition \ref{max-reg}, we have
\begin{eqnarray*}
0 &=& \int_{\R^N}\frac{1}{1+|x|^\alpha}(\lambda u-A_pu)u|u|^{p-2}\,dx\\
&=& \int_{\R^N}\left(\frac{\lambda+c|x|^\beta}{1+|x|^\alpha}+\frac{b(N+\alpha -2)|x|^{\alpha -2}+b(N-2)|x|^{2\alpha -2}}{p(1+|x|^\alpha)^2}\right)|u|^p\,dx \\
& & \quad +(p-1)\int_{\R^N}|\nabla u|^2 |u|^{p-2}\,dx.
\end{eqnarray*}
Hence, $u\equiv 0$.

Let now $f\in L^p(\mathbb{R}^N)$ and $u(x)=Lf(x)$ defined by \eqref{eq:def-L}.
Applying Lemma \ref{chius2} with $\gamma=0$,
we have $u\in L^p(\mathbb{R}^N)$. Moreover $u_n:=Lf_n$ satisfies $\lambda u_n-A_pu_n=f_n$ for any $f_n\in C_c^\infty(\mathbb{R}^N)$ approximating $f$ in $L^p(\mathbb{R}^N)$. Thus, $\lim_{n\to \infty}\|u_n-u\|_p=0$. Since, by local elliptic regularity, $A_p$ on $D_{p,max}(A)$ is closed, it follows that
$u\in D_{p,max}(A)$ and $\lambda u-A_pu=f$.
Thus $\lambda-A_p$ is invertible and $(\lambda-A_p)^{-1}\in \mathcal{L}(L^p(\mathbb{R}^N))$ for all $\lambda \ge \lambda_0$.
\\
Finally, \eqref{eq:stima-apriori} follows from \eqref{eq:stima-L}.
\end{proof}

%
%
%
%
The following result shows that the resolvent in $L^p(\R^N)$ and $C_0(\R^N)$ coincides.
\begin{theorem}\label{risolvente}
Assume that $N>2$, $\beta>\alpha-2$ and $\alpha>2$. Then, for all $\lambda \ge \lambda_0$, $(\lambda-A_p)^{-1}$ is a positive operator on $L^p(\mathbb{R}^N)$. Moreover, if $f \in L^p(\mathbb{R}^N)\cap C_{0}(\mathbb{R}^N)$, then $(\lambda -A_p)^{-1}f = (\lambda-A)^{-1}f$.
\end{theorem}
\begin{proof}
The positivity of $(\lambda-A_p)^{-1}$ follows from Theorem \ref{invertibile} and the positivity of $L$.

For the second assertion take
$f \in C_c^\infty(\mathbb{R}^N)$ and set $u:=(\lambda-A_p)^{-1}f$. Since the coefficients of $A$ are H\"older continuous, by local elliptic regularity (cf. \cite[Theorem 9.19]{DG-NT}), we know
$u \in C^{2+\sigma}_{loc}(\mathbb{R}^N)$ for some
$0<\sigma <1$. On the other hand, $u\in W^{2,p}(\mathbb{R}^N)$ by Proposition \ref{max-reg}.\\
If $p\ge \frac{N}{2}$ then, by Sobolev's
inequality, $u\in L^q(\mathbb{R}^N)$ for all $q\in [p,+\infty)$. In particular, $u\in L^q(\mathbb{R}^N)$ for some $q> \frac{N}{2}$ (cf. \cite[Corollary 9.13]{Brezis}) and
hence $Au=-f+\lambda u\in L^q(\mathbb{R}^N)$. Moreover, since $u \in C^{2+\sigma}_{loc}(\mathbb{R}^N)$ it follows that $u\in
W^{2,q}_{loc}(\mathbb{R}^N)$. So, $u\in D_{q,max}(A)\subset W^{2,q}(\mathbb{R}^N)\subset C_0(\mathbb{R}^N)$  by Proposition \ref{max-reg} and Sobolev's embedding theorem (cf. \cite[Corollary 9.13]{Brezis}).\\
 Let us now suppose that $p<\frac{N}{2}$. Take the sequence $(r_n)$, defined by $r_n=1/p-2n/N$ and set
$q_n=1/r_n$ for $n\in\N$. Let $n_0$ be the smallest integer such that $r_{n_0}\le 2/N$ noting that $r_{n_0}>0$.
Then, $u\in D_{p,max}(A)\subset L^{q_1}(\mathbb{R}^N)\cap L^p(\mathbb{R}^N)$,
by the Sobolev embedding theorem. As above we obtain that $u\in D_{q_1,max}(A)\subset L^{q_2}(\mathbb{R}^N)$. Iterating this
argument, we deduce that $u\in D_{q_{n_0},max}(A)$. So we can conclude that $u\in C_0(\mathbb{R}^N)$ arguing as in the previous
case. Thus,
$Au=-f+\lambda u\in C_b(\mathbb{R}^N)$.
Again, since $u \in C^{2+\sigma}_{loc}(\mathbb{R}^N)$, it follows that $u\in W^{2,q}_{loc}(\mathbb{R}^N)$ for any $q\in (1,+\infty)$.
Hence, $u\in D_{max}(A)$. So, by the uniqueness of the solution of the elliptic problem, we have $(\lambda
-A_p)^{-1}f=(\lambda -A)^{-1}f$
for every $f\in C_c^\infty(\mathbb{R}^N)$. Thus the assertion follows by density.

 \end{proof}

\section{Characterization of the domain and generation of semigroups}\label{section4}
The aim of this section is to prove that the operator $A_p$ generates an analytic semigroup on $L^p(\mathbb{R}^N)$, for any  $p \in (1,\infty)$, provided that
$N>2,\,\alpha > 2$ and $\beta>\alpha-2$.\\
We characterize first the domain of the operator $A_p$. More precisely we prove that the maximal domain $D_{p,max}(A)$ coincides with the weighted  Sobolev space $ D_p(A)$ defined by
\[
  D_p(A):=\{ u\in W^{2,p}(\mathbb{R}^N)\;:\; Vu,\ (1+|x|^{\alpha-1})\nabla u,\ (1+|x|^{\alpha}) D^2u\in L^p(\mathbb{R}^N)\}.
\]
%

In the following lemma we give a complete proof of the weighted gradient and second derivative estimates.

\begin{lemma}\label{lm:stima-gradiente}
Suppose that $N>2,\,\alpha >2$ and $\beta >\alpha -2$.
Then there exists a constant $C>0$ such that for every $u\in D_{p}(A)$ we have
\begin{equation}\label{eq:gradient-estimate}
\|(1+|x|^{\alpha-1})\nabla u\|_p\leq C (\|A_pu\|_p + \Vert u\Vert_p)\;,
\end{equation}
\begin{equation}\label{eq:second-derivative-estimate}
\|(1+|x|^{\alpha})D^2 u\|_p\leq C( \|A_pu\|_p + \Vert u\Vert_p)\;.
\end{equation}
\end{lemma}
\begin{proof}
Let $u\in D_{p}(A)$.
We fix $x_0\in \mathbb{R}^N$ and choose $\vartheta\in C_c^{\infty }(\mathbb{R}^N)$ such that $0\leq \vartheta \leq 1$, $\vartheta(x)=1$
for $x\in B(1)$ and $\vartheta(x)=0$ for $x\in \mathbb{R}^N \setminus B(2)$. Moreover, we set $\vartheta_\rho(x)=\vartheta
\left(\frac{x-x_0}{\rho}\right)$, where $\rho=\frac{1}{4}(1+|x_0|)$.
We apply the well-known interpolation inequality (cf. \cite[Theorem 7.27]{DG-NT})
\begin{equation}\label{interpolation}
\|\nabla v\|_{L^p(B(R))}\leq C\|v\|^{1/2}_{L^p(B(R))}
\|\Delta v\|^{1/2}_{L^p(B(R))},\;\;v\in W^{2,p}(B(R))\cap W^{1,p}_0 (B(R)),\;R>0,
\end{equation}
to the function $\vartheta_\rho u$ and obtain for every $\varepsilon>0$,
\begin{align*}
&\|(1+|x_0|)^{\alpha -1}\nabla u\|_{L^p(B(x_0,\rho))} \leq \|(1+|x_0|)^{\alpha -1}\nabla (\vartheta _\rho
u)\|_{L^p(B(x_0,2\rho))}\\
&\quad    \leq C \|(1+|x_0|)^{\alpha} \Delta (\vartheta _\rho u)\|^{\frac{1}{2}}_{L^p(B(x_0,2\rho))}
      \|(1+|x_0|)^{\alpha -2}\vartheta _\rho u\|^{\frac{1}{2}}_{L^p(B(x_0,2\rho))}\\
& \quad \leq C \left(
    \varepsilon \|(1+|x_0|)^{\alpha} \Delta (\vartheta _\rho u)\|_{L^p(B(x_0,2\rho))}+
  \frac{1}{4\varepsilon}\|(1+|x_0|)^{\alpha -2}\vartheta _\rho u\|_{L^p(B(x_0,2\rho))}
 \right)\\
&\quad \leq C \Big(
    \varepsilon \|(1+|x_0|)^{\alpha} \Delta u\|_{L^p(B(x_0,2\rho))}
     +\frac{2M}{\rho}\varepsilon \|(1+|x_0|)^{\alpha}\nabla u\|_{L^p(B(x_0,2\rho))} \\
&\qquad
    +\frac{\varepsilon M}{\rho^2}\| (1+|x_0|)^{\alpha}u\|_{L^p(B(x_0,2\rho))}
	  +\frac{1}{4\varepsilon}\|(1+|x_0|)^{\alpha -2} u\|_{L^p(B(x_0,2\rho))}
 \Big)\\
&\quad \leq C \Big(
    \varepsilon \|(1+|x_0|)^{\alpha} \Delta u\|_{L^p(B(x_0,2\rho))}
     +8M \varepsilon \|(1+|x_0|)^{\alpha-1}\nabla u\|_{L^p(B(x_0,2\rho))}\\
&\qquad +\left( 16\varepsilon M+\frac{1}{4\varepsilon}\right) \|(1+|x_0|)^{\alpha -2} u\|_{L^p(B(x_0,2\rho))}
      \Big)\\
&\quad \leq C(M) \Big(
    \varepsilon \|(1+|x_0|)^{\alpha} \Delta u\|_{L^p(B(x_0,2\rho))}
     + \varepsilon \|(1+|x_0|)^{\alpha-1}\nabla u\|_{L^p(B(x_0,2\rho))}\\
&\qquad + \frac{1}{\varepsilon} \|(1+|x_0|)^{\alpha -2} u\|_{L^p(B(x_0,2\rho))}
      \Big),
\end{align*}
where $M=\|\nabla \vartheta \|_{\infty }+\|\Delta \vartheta\|_{\infty }$.
Since  $2\rho=\frac{1}{2}(1+|x_0|)$ we get
\[
\frac{1}{2}(1+|x_0|)\leq 1+|x|\leq \frac{3}{2}(1+|x_0|),\qquad x\in B(x_0,2\rho).
\]
Thus,
\begin{align}\label{eq:cover-x0}
&\|(1+|x|)^{\alpha -1}\nabla u\|_{L^p(B(x_0,\rho))}\leq \left(\frac 32 \right)^{\alpha -1}\|(1+|x_0|)^{\alpha -1}\nabla
u\|_{L^p(B(x_0,\rho))} \nonumber \\
& \quad \leq C\Big(
    \varepsilon \|(1+|x_0|)^{\alpha} \Delta u\|_{L^p(B(x_0,2\rho))}
     + \varepsilon \|(1+|x_0|)^{\alpha-1}\nabla u\|_{L^p(B(x_0,2\rho))}\nonumber \\
&\qquad +\frac{1}{\varepsilon} \|(1+|x_0|)^{\alpha -2}u\|_{L^p(B(x_0,2\rho))}
      \Big)\;\nonumber \\
& \quad \leq C \Big(
    2^\alpha \varepsilon \|(1+|x|)^{\alpha} \Delta u\|_{L^p(B(x_0,2\rho))}
     +2^{\alpha-1} \varepsilon \|(1+|x|)^{\alpha-1}\nabla u\|_{L^p(B(x_0,2\rho))}\nonumber \\
&\qquad +\frac{2^{\alpha-2}}{\varepsilon} \|(1+|x|)^{\alpha -2}u\|_{L^p(B(x_0,2\rho))}
      \Big).
\end{align}
Let $\{B(x_n,\rho(x_n ))\}$ be a countable covering of $\mathbb{R}^N$ as in Proposition \ref{pr:covering}
such that at most $\zeta $ among the double balls $\{B(x_n,2\rho(x_n ))\}$ overlap.

We write \eqref{eq:cover-x0} with $x_0$ replaced by $x_n$ and sum over $n$, we obtain
\begin{eqnarray*}
\|(1+|x|)^{\alpha -1}\nabla u\|_{p}& \leq & \ C\zeta
    \big( \varepsilon \|(1+|x|)^{\alpha} \Delta u\|_{p}
     +\varepsilon \|(1+|x|)^{\alpha-1}\nabla u\|_{p}+\frac{1}{\varepsilon} \|(1+|x|)^{\alpha -2}u\|_{p}
     \big)\;.\nonumber \\
     &\le & C\varepsilon\Vert A_p u\Vert_p + C\varepsilon(1+\vert b\vert)\|(1+|x|)^{\alpha -1}\nabla u\|_{p} + C(\frac{1}{\varepsilon} + \varepsilon)\Vert (1+\vert x\vert^{\beta})u\Vert_p.
\end{eqnarray*}
Choosing $\varepsilon$ such that $\varepsilon C\zeta<\frac{1}{2(1+\vert b\vert)}$ we have
$$
\|(1+|x|)^{\alpha -1}\nabla u\|_{p} \leq C\left( \|A_p u\|_{p}+\Vert (1+\vert x\vert^{\beta})u\Vert_p\right)
$$
for some constant $C>0$.
Furthermore, by \eqref{eq:stima-apriori}, we know that $ \|(1+|x|^\beta)u\|_p \leq C(\|A_pu\|_p + \Vert u\Vert_p)$
for every $u\in D_p(A)\subset D_{p,max}(A)$ and some $C>0$. Hence,
\begin{equation*}
\|(1+|x|)^{\alpha -1}\nabla u\|_{p} \leq  C(\|A_pu\|_p+\|u\|_p).
\end{equation*}
As regards the second order derivatives we recall
the classical Calder\'{o}n-Zygmund inequality on $B(1)$
$$
\|D^2v \|_{L^p(B(1))} \le C\|\Delta v \|_{L^p(B(1))}
,\;\;v\in W^{2,p}(B(1))\cap W^{1,p}_0 (B(1)).
$$
By rescaling and translating we obtain
\begin{equation}\label{eq:cal-zig}
\|D^2v \|_{L^p(B(x_0,R))} \le C\|\Delta v \|_{L^p(B(x_0,R))}
\end{equation}
for every $x_0\in \mathbb{R}^N$, $R>0$ and
$v\in W^{2,p}(B(x_0,R))\cap W^{1,p}_0 (B(x_0,R))$.
We observe that the constant $C$ does not depend on $R$ and $x_0$.
\\
Then we fix $x_0\in \mathbb{R}^N$ and choose $\rho $ and $\vartheta_{\rho}\in C_c^{\infty }(\mathbb{R}^N)$ as above.
Applying \eqref{eq:cal-zig} to the function $\vartheta_{\rho} u$ in $B(x_0,2\rho)$, we obtain
\begin{align*}
&\|(1+|x_0|)^{\alpha}D^2 u\|_{L^p(B(x_0,\rho))}
  \leq \|(1+|x_0|)^{\alpha }D^2 (\vartheta _\rho u)\|_{L^p(B(x_0,2\rho))}\\
&\quad    \leq C \|(1+|x_0|)^{\alpha} \Delta (\vartheta _\rho u)\|_{L^p(B(x_0,2\rho))}.
\end{align*}
Arguing as above we obtain
\begin{align*}
&\|(1+|x|)^{\alpha }D^2 u\|_p
      \leq C \left(
    \|(1+|x|)^{\alpha} \Delta u\|_p
     +  \|(1+|x|)^{\alpha-1}\nabla u\|_p
     +\|(1+|x|)^{\alpha -2} u\|_p
      \right).
\end{align*}
The lemma follows from \eqref{eq:stima-apriori} and \eqref{eq:gradient-estimate}.
\end{proof}

The following result shows that $C_{c}^{\infty}(\mathbb{R}^N)$ is a core for $A_p$, since by Lemma \ref{lm:stima-gradiente} the norm \eqref{norm*}
is equivalent to the graph norm of $A_p$. The proof is based on Theorem \ref{invertibile} and Lemma \ref{lm:stima-gradiente} and it is similar to the one given in \cite[Lemma 4.3]{AC-AR-CT}.
\begin{lemma}\label{density1}
The space $C_{c}^{\infty}(\mathbb{R}^N)$ is dense in
\begin{equation}
D_{p}(A) = \{ u\in  W^{2,p}(\mathbb{R}^N), Vu, (1+\vert x\vert^{\alpha})D^{2}u, (1+\vert x\vert^{\alpha-1})\nabla u \in L^{p}(\mathbb{R}^{N})\}
\end{equation}
endowed with the norm
\begin{equation}\label{norm*}
\|u\|_{D_p(A)}:=\|u\|_p+\|Vu\|_p+\|(1+|x|^{\alpha -1})|\nabla u|\|_p+ \|(1+|x|^{\alpha})|D^2u|\|_p,\quad u\in D_p(A).
\end{equation}
\end{lemma}

Now, we are ready to show the main result of this section:
\begin{theorem}\label{generation}
Suppose that $N>2,\,\alpha>2$ and $\beta>\alpha-2$. Then the operator $A_p$ with domain $D_{p,max}(A)$ generates an analytic semigroup in $L^p(\mathbb{R}^N)$.
\end{theorem}
\begin{proof}
Let $f \in L^{p}(\mathbb{R}^N),\, \rho>0$. Consider the operator $\widehat{A_{p}} := A_{p}-\omega$, where $\omega$ is a constant which will be chosen later. It is known that the elliptic problem in $L^{p}(B(\rho))$
\begin{equation}\label{equation 4}
\begin{cases}
& \lambda u - \widehat{A_{p}} u = f \ \ in\ \ B(\rho),\\
& u = 0 \qquad  on\ \ \partial B(\rho)
\end{cases}
\end{equation}
admits a unique solution $u_{\rho}$ in $W^{2,p}(B(\rho))\cap W_{0}^{1,p}(B(\rho))$ for $\lambda>0$, (cf. \cite[Theorem 9.15]{DG-NT}).\\
Let us prove that $e^{\pm i\theta}\widehat{A_{p}}$ is dissipative in $B(\rho)$ for $0 \le \theta \le \theta_{\alpha}$ with suitable $\theta_{\alpha} \in (0,\frac{\pi}{2}]$. To this purpose observe that
$$
\widehat{A_{p}}u_{\rho}) = \text{div}((1+\vert x\vert^{\alpha})\nabla u_{\rho}) + (b-\alpha)\vert x \vert^{\alpha - 1}\frac{x}{\vert x \vert} \cdot \nabla u_{\rho} - \vert x\vert^{\beta} u_{\rho} - \omega u_{\rho}.
$$
Set $u^{*} = \bar{u}_{\rho}\vert u_{\rho}\vert^{p-2}$ and recall that $q(x) = 1 + \vert x \vert^{\alpha}$. Multiplying $\widehat{A_{p}}u_{\rho}$ by $u^{*}$ and integrating over $B(\rho)$, we obtain
\begin{align*}
& \int_{B(\rho)}\widehat{A_{p}}u_{\rho}u^{*} dx = - \int_{B(\rho)}q(x)\vert u_{\rho}\vert^{p-4}\vert Re(\bar{u}_{\rho}\nabla u_{\rho})\vert^{2}dx  - \int_{B(\rho)}q(x)\vert u_{\rho}\vert^{p-4}\vert Im(\bar{u}_{\rho}\nabla u_{\rho})\vert^{2}dx \\ &- \int_{B(\rho)}\bar{u}_{\rho}\vert u_{\rho}\vert^{p-2}\nabla q(x)\nabla u_{\rho}dx - (p-2)\int_{B(\rho)}q(x)\vert u_{\rho}\vert^{p-4}\bar{u}_{\rho}\nabla u_{\rho}\vert Re(\bar{u}_{\rho}\nabla u_{\rho})\vert^{2}dx \\
& + b\int_{B(\rho)}\bar{u}_{\rho} \vert u_{\rho}\vert^{p-2} \vert x \vert^{\alpha-1}\frac{x}{\vert x \vert}\nabla u_{\rho} dx - \int_{B(\rho)}\left(\vert x\vert^{\beta} + \omega \right)\vert u_{\rho}\vert^{p}dx.
\end{align*}
We note here that the integration by part in the singular case $1<p<2$ is allowed thanks to \cite{Me-Sp}.
By taking the real part of the left and the right hand side, we have
\begin{align*}
& Re\left( \int_{B(\rho)}\widehat{A_{p}}u_{\rho}u^{*} dx\right) \\
 = & -(p-1)\int_{B(\rho)}q(x)\vert u_{\rho}\vert^{p-4}\vert Re(\bar{u}_{\rho}\nabla u_{\rho})\vert^{2}dx  - \int_{B(\rho)}q(x)\vert u_{\rho}\vert^{p-4}\vert Im(\bar{u}_{\rho}\nabla u_{\rho})\vert^{2}dx\\ & -
 \int_{B(\rho)}\vert u_{\rho}\vert^{p-2}\nabla q(x)Re(\bar{u}_{\rho}\nabla u_{\rho})dx
  + b\int_{B(\rho)} \vert u_{\rho}\vert^{p-2} \vert x \vert^{\alpha-1}\frac{x}{\vert x \vert}Re(\bar{u}_{\rho}\nabla u_{\rho}) dx\\
 & - \int_{B(\rho)}\left(\vert x\vert^{\beta} + \omega \right)\vert u_{\rho}\vert^{p}dx. \\
 & = -(p-1)\int_{B(\rho)}q(x)|u_\rho|^{p-4}|Re(\bar{u}_\rho\nabla u_\rho)|^2dx -\int_{B(\rho)}q(x)|u_\rho|^{p-4}|Im(\bar{u}_\rho\nabla u_\rho)|^2dx\\
&+ \int_{B(\rho)} \left( \frac{ (\alpha-b)(N-2+\alpha) }{p}|x|^{\alpha-2}-|x|^\beta-\omega \right)|u_\rho|^pdx.
\end{align*}
Taking now the imaginary part of the left and the right hand side, we obtain
\begin{align*}
& Im\left( \int_{B(\rho)}\widehat{A_{p}}u_{\rho}u^{*} dx\right) \\
 = & -(p-2)\int_{B(\rho)}q(x)\vert u_{\rho}\vert^{p-4} Im(\bar{u}_{\rho}\nabla u_{\rho})Re(\bar{u}_{\rho}\nabla u_{\rho})dx \\ & -
 \int_{B(\rho)}\vert u_{\rho}\vert^{p-2}\nabla q(x)Im(\bar{u}_{\rho}\nabla u_{\rho})dx
  + b\int_{B(\rho)} \vert u_{\rho}\vert^{p-2} \vert x \vert^{\alpha-1}\frac{x}{\vert x \vert}Im(\bar{u}_{\rho}\nabla u_{\rho}) dx.\\
 \end{align*}
 We can choose $\omega >0$ such that
$$
 \frac{ (\alpha-b)(N-2+\alpha) }{p}|x|^{\alpha-2}-|x|^\beta-\omega \le -\frac{ \vert \alpha-b\vert (N-2+\alpha) }{p} \vert x\vert^{\alpha-2}.
$$
Furthermore,
\begin{eqnarray*}
-Re\left( \int_{B(\rho)}\widehat{A_{p}}u_\rho\, u^\star dx  \right)
&\geq &(p-1)\int_{B(\rho)}q(x)|u_\rho|^{p-4}|Re(\bar{u}_\rho\nabla u_\rho)|^2dx
    \\ & &
  +\int_{B(\rho)}q(x)|u_\rho|^{p-4}|Im(\bar{u}_\rho\nabla u_\rho)|^2dx
+\tilde c \int_{B(\rho)}|u_\rho|^{p}|x|^{\alpha-2}dx
\\ &=&
(p-1)B^2+C^2+\tilde c D^2,
\end{eqnarray*}
where $\tilde c = \frac{ \vert \alpha-b\vert (N-2+\alpha) }{p}$ is a positive constant.
\\
Moreover,
\begin{eqnarray*}
& & \left|Im \left(\int_{B(\rho)}\widehat{A_{p}} u_\rho\, u^\star dx\right)\right|\\
&\leq & |p-2|\left(\int_{B(\rho)}|u_\rho|^{p-4}q(x)|Re(\bar{u}_\rho\nabla u_\rho)|^2 dx\right)^\frac{1}{2}
    \left(\int_{B(\rho)}|u_\rho|^{p-4}q(x)|Im(\bar{u}_\rho\nabla u_\rho)|^2 dx\right)^\frac{1}{2}\\
& &\qquad
  +\vert \alpha-b\vert \left(\int_{B(\rho)}|u_\rho|^{p-4}|x|^{\alpha}|Im(\bar{u}_\rho\nabla u_\rho)|^2\;
  dx\right)^\frac{1}{2}\left(\int_{B(\rho)}|u_\rho|^{p}|x|^{\alpha-2}\; dx\right)^\frac{1}{2}\\
&=& |p-2|BC+\vert \alpha-b \vert CD.
\end{eqnarray*}
Setting
\begin{align*}
& B^{2} =  \int_{B(\rho)}q(x)\vert u_{\rho}\vert^{p-4}\vert Re(\bar{u}_{\rho}\nabla u_{\rho})\vert^{2} dx\\
& C^{2} =  \int_{B(\rho)}q(x)\vert u_{\rho}\vert^{p-4}\vert Im(\bar{u}_{\rho}\nabla u_{\rho})\vert^{2}dx \\
& D^{2} =  \int_{B(\rho)}  \vert x \vert^{\alpha-2}\vert u_{\rho} \vert^{p}dx.
\end{align*}
As a result of the above estimates, we conclude
$$
\left\vert Im\left( \int_{B(\rho)}\widehat{A_{p}}u_{\rho}u^{*} dx\right) \right\vert \le l_{\alpha}^{-1} \left[- Re\left( \int_{B(\rho)}\widehat{A_{p}}u_{\rho}u^{*}dx\right)\right].
$$
If $\tan\theta_{\alpha} = l_{\alpha}$, then $e^{\pm i\theta}\widehat{A_{p}}$ is dissipative in $B(\rho)$ for $0\le \theta \le \theta_{\alpha}$ . From \cite[Theorem I.3.9]{Pazy} it follows that the problem
(\ref{equation 4}) has a unique solution $u_\rho$ for every $\lambda \in
\Sigma_\theta ,\,0 \le \theta <\theta_\alpha$ where
$$
\Sigma_\theta =\{\lambda \in \mathbb{C}\setminus \{0\}: |Arg\,  \lambda| <
\pi/2+\theta\}.
$$
Moreover, there exists a constant $C_\theta$ which is independent of $\rho$, such that
\begin{equation} \label{stima}
\|u_\rho\|_{L^p (B(\rho))} \le
\frac{C_\theta}{|\lambda|}\|f\|_{L^p},\quad \lambda \in \Sigma_\theta .
\end{equation}
Let us now fix  $\lambda \in \Sigma_\theta$, with $0<\theta <\theta_\alpha$ and a radius $r>0$.
We apply the interior $L^p$ estimates
(cf. \cite[Theorem 9.11]{DG-NT})  to the functions $u_\rho$ with $\rho >r+1$. So, by
(\ref{stima}), we have
\begin{equation}\label{eq:stimaW2prho}
\|u_\rho\|_{W^{2,p}(B(r))}\leq C_1\left( \|\lambda
u_\rho-\widehat{A_{p}}u_\rho\|_{L^p(B(r+1))}+\|u_\rho\|_{L^p(B(r+1))}\right) \le C_2\|f\|_p.
\end{equation}
Using a weak compactness and a diagonal argument,
we can construct a sequence $(\rho_n) \to \infty$ such that the functions $(u_{\rho_n})$
converge weakly in $W^{2,p}_{loc}$ to a function $u$ which satisfies $\lambda u-\widehat{A_{p}}u=f$ and
\begin{equation}
\label{stima-2}
 \|u\|_{p} \le
\frac{C_\theta}{|\lambda|}\|f\|_{p}, \quad \lambda \in \Sigma_\theta .
\end{equation}
Moreover, $u \in D_{p,max}(A)$.
We have now only to show
that $\lambda-\widehat{A_{p}}$ is invertible on $D_{p,max}(A)$ for
$\lambda_0 <\lambda \in \Sigma_\theta$.
Consider the set
\[
E=\{r>0: \Sigma_\theta \cap C(r) \subset\rho (\widehat{A_{p}})\},
\]
where $C(r):=\{\lambda \in \C : |\lambda |<r\}$.
Since, by Theorem \ref{invertibile}, $\lambda_0$ is in the resolvent set of $\widehat{A_{p}}$, then
$R=\sup E>0$.
On the other hand, the norm of
the resolvent is bounded by
$C_\theta/|\lambda|$ in $C(R) \cap \Sigma_\theta$. Consequently  it cannot explode on the
boundary of $C(R)$. Then $R=\infty$ and this ends the proof of the theorem.
\end{proof}

Let us show that $D_{p,max}(A)$ and $D_p(A)$ coincide.
\begin{proposition}\label{th:domain}
Assume that $N>2,\,\alpha >2$ and $\beta >\alpha -2$. Then maximal domain $D_{p,max}(A)$  coincides with $D_p(A)$.
\end{proposition}
\begin{proof}
We have to prove only the inclusion $D_{p,max}(A)\subset D_{p}(A)$.

Let $\tilde u\in D_{p,max}(A)$ and set $f=\lambda\,\tilde u - A_p\tilde u$.
The operator $A$ in $B(\rho)$, $\rho>0$,  is an uniformly elliptic operator with bounded coefficients. Then  the Dirichlet problem
\begin{equation}\label{palla0}
\left\{
\begin{array}{ll}
\lambda\,u- Au=f&\text{ in }B(\rho)\\
u=0&\text{ on }\partial B(\rho)\;,
\end{array}
\right.
\end{equation}
admits a unique solution $u_\rho$ in $W^{2,p}(B(\rho))\cap W_0^{1,p}(B(\rho))$ (cf. \cite[Theorem 9.15]{DG-NT}). So, $\widetilde{u_\rho}$, the zero extension
of $u_\rho$ to the complement $B(\rho)^c$, belongs to $D_p(A)$. Thus,
by Lemma \ref{lm:stima-gradiente} and \eqref{eq:stima-apriori}, we have
\begin{align*}
&\|(1+|x|^{\alpha-2})\widetilde{u_\rho}\|_p+\|(1+|x|^{\alpha-1})\nabla \widetilde{u_\rho}\|_p\\
&\qquad +\|(1+|x|^{\alpha})D^2 \widetilde{u_\rho}\|_p+\|V\widetilde{u_\rho}\|_p \le C(\|A\widetilde{u_\rho}\|_{p} + \Vert \widetilde{u_{\rho}}\Vert_p)
\end{align*}
with $C$ independent of $\rho$.

We observe that $u_\rho$ is the solution of \eqref{equation 4} with $\lambda$ replaced with $\lambda-\omega$.
Then arguing as in the proof of Theorem \ref{generation},
by \eqref{stima} and \eqref{eq:stimaW2prho} for $\lambda>\omega$, we have
$\|u_\rho\|_{L^p (B(\rho))} \le \frac{C_1}{\lambda-\omega}\|f\|_{L^p}$
and
$\|u_\rho\|_{W^{2,p}(B(r))}\leq C_2\|f\|_{L^p}$ where $r<\rho-1 $ and $C_1,C_2$ are positive constants which do not depend on $\rho$.

Using a standard weak compactness argument we can construct a sequence
$\widetilde{u_{\rho_n}}$ which converges to a function $u$ in $W^{2,p}_{loc}(\R^N)\cap L^p(\R^N)$ such that $\lambda\,u -Au=f$.
Since the estimates above are independent of $\rho$, also $u\in D_p(A)$.
Then $\lambda \tilde u - A\tilde u=\lambda u- Au$ and since $D_p(A)\subset D_{p,max}(A)$ and $\lambda - A$ is invertible on $D_{p,\max}(A)$ by
Theorem \ref{invertibile}, we have $\tilde u=u$.
\end{proof}
\begin{proposition}
For any $f \in L^p(\mathbb{R}^N)$, $1<p<\infty$, and any $0<\nu<1$ and for all $t >0$, the function $T_p(t)f$ belongs to $C_{b}^{1+\nu}(\mathbb{R}^N)$. In particular, the semigroup $(T_p(t))_{t\ge 0}$ is ultracontractive.
\end{proposition}
\begin{proof}
In Theorem \ref{generation} we have proved that $A_p$ generates an analytic semigroup $T_p(\cdot)$ on $L^p(\R^N)$ and in Theorem \ref{risolvente} we have obtained that for $f \in L^p(\mathbb{R}^N)\cap C_0(\mathbb{R}^N)$, $(\lambda -A_p)^{-1}f = (\lambda-A)^{-1}f$. Hence this shows the coherence of the resolvents on $L^p(\mathbb{R}^N)\cap L^q(\mathbb{R}^N)$ by using a density argument. This will yield immediately that the semigroups are coherent in different $L^p$-spaces. One can deduce the result by using the same arguments as in the proof of \cite[Proposition 2.6]{Luca - Abde}.
\end{proof}
To end this section we study the spectrum of $A_p$.
\begin{proposition}
Assume $N>2$, $\alpha>2$, $\beta>\alpha-2$. Then, for $p \in (1,\infty)$, the resolvent operator $R(\lambda,A_p)$ is compact in $L^p(\mathbb{R}^N)$ for all $\omega_0 <\lambda \in \rho(A_p)$, where $\omega_0$ is a suitable positive constant, and the spectrum of $A_p$ consists of a sequence of negative real eigenvalues which accumulates at $-\infty$. Moreover, $\sigma(A_p)$ is independent of $p$.
\end{proposition}
\begin{proof}
The proof is similar to the one given in \cite{Luca - Abde}.
\end{proof}

$$ $$
\textbf{Acknowledgement: }
S-E. Boutiah wishes to thank the Dept. of Information Eng., Electrical Eng. and Applied Mathematics (D.I.E.M) of the University of Salerno for the warm hospitality and a very fruitful and pleasant stay in Salerno, where this paper has been written.


\begin{thebibliography}{99}

\bibitem{Au-Ben} {P. Auscher, B. Ben Ali}, \emph{Maximal inequalities and Riesz transform estimates on $L^p$ spaces for
Schr\"odinger operators with nonnegative potentials}, Ann. Inst. Fourier {\bf 57} (2007), 1975-2013.
\bibitem{Brezis} {H. Brezis}, \emph{Functional Analysis, Sobolev Spaces and Partial Differential Equations}, Springer, 2011.
\bibitem{AC-AR-CT} {A. Canale, A. Rhandi, C. Tacelli}, \emph{Schr\"odinger-type operators with unbounded diffusion and potential terms}, Ann. Sc. Norm. Super. Pisa CI. Sci. (5) Vol. XVI (2016), 581-601.
\bibitem{cup-for} {G. Cupini, S. Fornaro}, \emph{Maximal regularity in $L^p$ for a class of elliptic operators with unbounded coefficients}, Diff. Int. Eqs. \textbf{17} (2004), 259-296.
\bibitem{F-L} {S. Fornaro, L. Lorenzi}, \emph{Generation results for elliptic operators with unbounded diffusion coefficients in $L^{p}$ and $C_{b}$-spaces}, Discrete and Continuous Dynamical Systems A {\bf 18} (2007), 747-772.
\bibitem{DG-NT} {D. Gilbarg, N. Trudinger}, \emph{Elliptic Partial Differential Equations of Second Order}, Second edition, Springer, Berlin, (1983).
\bibitem{TD-RM-CT} {Tiziana Durante, Rosanna Manzo, Cristian Tacelli}, \emph{Kernel estimates for Schr\"odinger type operators with unbounded coefficients and critical exponent}, Ricerche Mat. {\bf 65} (2016), 289-305.
\bibitem{Lo-Be} {L. Lorenzi, M. Bertoldi}, \emph{Analytical Methods for Markov Semigroups}, Chapman $ \&$ Hall/CRC, (2007).
\bibitem{Luca - Abde} {L. Lorenzi, A. Rhandi}, \emph{On Schrodinger type operators with unbounded coefficients: generation and heat kernel estimates}, J. Evol. Equ. {\bf 15} (2015), 53-88.
\bibitem{GM-NO-MS-CP} {G. Metafune, N. Okazawa, M. Sobajima, C. Spina}, \emph{Scale invariant elliptic operators with singular coefficients}, J. Evol. Equ. {\bf 16} (2016), 391-439.
\bibitem{Me-Pa-Wa} {G. Metafune, D. Pallara, M. Wacker}, \emph{Feller Semigroups on $\mathbb{R}^{N}$}, Semigroup Forum {\bf 65} (2002), 159-205.
\bibitem{Me-Sp} {G. Metafune, C. Spina}, \emph{An integration by parts formula in Sobolev spaces}, Mediterranean Journal of Mathematics {\bf 5} (2008), 359-371.
\bibitem{G-S} {G. Metafune, C. Spina}, \emph{Elliptic operators with unbounded coefficients in $L^{p}$ spaces}, Annali Scuola Normale Superiore di Pisa Cl. Sc. (5), {\bf 11} (2012), 303-340.
\bibitem{G-S 2} {G. Metafune, C. Spina}, \emph{A degenerate elliptic operators with unbounded coefficients}, Rend. Lincei Mat. Appl. {\bf 25} (2014), 109-140.
\bibitem{G-S 3} {G. Metafune, C. Spina}, \emph{Kernel estimates for some elliptic operators with unbounded coefficients}, Discrete and Continuous Dynamical Systems A {\bf 32} (2012), 2285-2299.
\bibitem{Me-Sp-Ta} {G. Metafune, C. Spina, C. Tacelli}, \emph{Elliptic operators with unbounded diffusion and drift coefficients in $L^{p}$ spaces}, Adv. Diff. Equat. {\bf 19} (2014), 473-526.
\bibitem{okazawa} {N.Okazawa}, \emph{An $L^p$ theory for Schrodinger operators with nonnegative potentials}, J. Math. Soc. Japan {\bf 36} (1984), 675-688.
\bibitem{Pazy} {A. Pazy}, \emph{Semigroups of linear operators and applications to partial differential equations}, Applied mathematical sciences 44, Springer-Verlag, 1983.
\bibitem{shen-1995} {Z. Shen}, \emph{$L_p$ estimates for Schr\"odinger operators
with certain potentials}, Ann. Inst. Fourier \textbf{45} (1995), 513--546.

\end{thebibliography}
\end{document}